\documentclass[a4paper,10pt]{amsart}
\usepackage[utf8]{inputenc}

\usepackage{graphicx,amssymb,amsmath,amsfonts,amsthm,amscd,hyperref,textcomp,relsize}
\usepackage{tikz-cd}
\usetikzlibrary{babel} 

\newtheorem{proposition}{Proposition}
\newtheorem{defn}{Definition}
\newtheorem{lemma}{Lemma}
\newtheorem{cor}{Corollary}

\theoremstyle{remark}

\newcommand{\CC}{\mathbb C}

\newcommand{\Ql}{\bar{\mathbb Q}_\ell}
\newcommand{\Dbc}{D^b_c}

\newcommand{\FF}{\mathcal F}
\newcommand{\GGG}{\mathcal G}
\newcommand{\HH}{\mathcal H}
\newcommand{\LL}{\mathcal L}
\newcommand{\R}{\mathrm R}

\newcommand{\Gal}{\mathrm{Gal}}

\newcommand{\AAA}{\mathbb A}
\newcommand{\GG}{\mathbb G}

\newcommand{\Tr}{\mathrm{Tr}}

\newcommand{\HHH}{\mathrm H}

\newcommand{\Spec}{\mathrm{Spec}}

\newcommand{\GL}{\mathrm{GL}}

\newcommand{\CCC}{\mathcal C}

\title{Tensor and convolution direct image of $\ell$-adic sheaves}
\author{Antonio Rojas-Le\'on}
\address{Dpto. de \'Algebra, Fac. de Matem\'aticas \\
  Universidad de Sevilla \\
  c/Tarfia, s/n \\
  41013 Sevilla, Spain \\
  \tt{arojas@us.es}}

\begin{document}

\maketitle

\renewcommand{\thefootnote}{}
\footnote{Partially supported by MTM2016-75027-P (Ministerio de Econom\'{\i}a y Competitividad) and FEDER}

\begin{abstract}
Given a Galois étale map of varieties $\pi:Y\to X$ and an $\ell$-adic sheaf or derived category object $P\in\Dbc(Y,\Ql)$, we study two cohomological operations: the tensor direct image and the convolution direct image (in the case of perverse sheaves), which give objects of $\Dbc(X,\Ql)$, and can be used to improve the estimates on some partial exponential sums and the number of rational points on certain varieties.
\end{abstract}

\section{Introduction}

Let $k$ be a finite field of cardinality $q=p^n$, and $X$ a commutative group variety defined over $k$. Fix a prime $\ell\neq p$, and let $k_r$ be the degree $r$ extension of $k$ in a fixed algebraic closure $\bar k$. Given an $\ell$-adic constructible sheaf $\FF$ on $X$ (or, more generally, an object $P$ in the derived category $\Dbc(X,\Ql)$), one can consider ``trace and norm Frobenius sums'' of the following form, for $a\in X(k)$:
$$
S(P,a)=\sum_{y\in X(k_r);\mathrm{N}_{X(k_r)/X(k)}(y)=x}\mathrm{N}(F_y|P_{\bar y})
$$
where $F_y$ is a geometric Frobenius element of $X$ at $y$ and $P_{\bar y}$ is the stalk of $P$ at a geometric point over $y$. These sums where studied in \cite{rl-rationality}, where we developed a cohomological construction (the \emph{convolution Adams operation}) that, given an object $P$, produces another object whose (regular) Frobenius traces are precisely the sums $S(P,a)$. This was used, for example, to give sharp estimates for the number of rational points on certain Artin-Schreier and superelliptic curves.

Now let $Y=X\times_{\Spec\;k}\Spec\;k_r$ be the extension of scalars, and suppose that $P$ is only defined on $Y$ (but not on $X$). Then the sums $S(P,a)$ still make sense, so one could ask whether they are given by the Frobenius traces of some sheaf or derived category object on $X$. The main goal of this article is to construct such an object, which will be called the \emph{convolution direct image} of $P$.

Given a finite index subgroup $H$ of a group $G$ and a representation $\rho$ of $H$ on some finitely generated free module, the \emph{tensor induction} of $\rho$ is a representation of $G$ derived from it. Roughly speaking, it is constructed in a similar way to the induced representation, but replacing the direct sum with the tensor product. See \cite[Section 13]{curtis-reiner} for details. Using the fact that the category of $\ell$-adic lisse sheaves on $X$ (respectively on $Y$) is equivalent to the category of continuous $\ell$-adic representations of its fundamental group $\pi_1(X,\bar x)$ with base point a given geometric point $\bar x$ of $X$ (resp. representations of $\pi_1(Y,\bar y)$) and, for a finite \'etale map $\pi:Y\to X$ such that $\pi(\bar y)=\bar x$ the group $\pi_1(Y,\bar y)$ is a finite index subgroup of $\pi_1(X,\bar x)$, this was used by N. Katz in \cite[10.3-6]{katz1990esa} to define the tensor direct image of a lisse sheaf on $Y$. In \cite[Proposition 2.1]{fu-wan-incomplete}, L. Fu and D. Wan showed that, if $Y=X\times\Spec\;k_r$ and $a\in X(k)$, the trace of a geometric Frobenius element $F$ at $a$ acting on the stalk at $a$ of the tensor direct image of a lisse sheaf $\FF$ is precisely the trace of $F^r$ acting on the stalk of $\FF$ at $a$. This was used to give estimates for some partial exponential sums. In this article we start by generalizing this construction to arbitrary constructible sheaves and derived category objects, so the number of Frobenius trace sums to which this result can be applied is greatly increased.

 In section 2 we start by recalling the definition and main properties of the tensor induction of representations, and we give an alternative characterization when $H$ is a normal subgroup of $G$. This characterization can be generalized to the category of constructible sheaves, and this is what we do in section 3. Given a finite Galois \'etale cover $\pi:Y\to X$, we define the tensor direct image $\pi_{\otimes\ast}\FF$ of a constructible \'etale sheaf $\FF$ on $Y$ with respect to $\pi$, and show that it generalizes the construction given in \cite{katz1990esa} for locally constant objects. We also show that, in the case $Y=X\times\Spec\;k_r$, its Frobenius traces are given by the same formula that was proved in \cite{fu-wan-incomplete} for lisse sheaves (Proposition \ref{frobenius}):
 $$
 \Tr(F_x|(\pi_{\otimes\ast}\FF)_{\bar x})=\Tr(F_x^r|\FF_{\bar x_r})
 $$

In section 4 we extend this construction to the derived category of \'etale sheaves on $Y$, and show that essentially the same properties hold for these objects. In section 5 we work on  commutative group schemes, and work out the same construction but replacing the tensor product operation with the convolution in the category of perverse sheaves. This gives rise to a new operation: the \emph{convolution direct image} of an perverse sheaf $P\in\mathcal Perv _Y$, which is a perverse sheaf on $X$. The main result is Proposition \ref{traceformula}: in the case $Y=X\times\Spec\;k_r$, the Frobenius traces of $\pi_{cv\ast}P$ on a point $x\in X(k)$ are given by
$$
 \Tr(F_x|\pi_{cv\ast}P_{\bar x})=\sum_{y\in X(k_r)|\Tr(y)=x}\Tr(F_y|P_{\bar y}).
$$

Finally, in the last section we breafly describe how this result can be applied to estimate exponential sums defined by trace constraints and the number of rational points on some types of curves, and we will compare these estimates with the ones obtained in \cite{rl-rationality} using convolution Adams power.

We would like to thank an anonymous referee for pointing out some errors in a previous version of this article.

\section{Tensor induction of representations}

Let us start by recalling the tensor induction operation for group representations \cite[Section 13]{curtis-reiner}. Let $G$ be a group, $H\subseteq G$ a subgroup of finite index, $A$ a ring and $\rho:H\to\GL(M)$ a representation of $H$ on the automorphism group of a finitely generated free $A$-module $M$. Fix a set $Hg_1,\ldots,Hg_d$ of representative for the right cosets of $H$ in $G$. Given an element $g\in G$, let $g_ig=h_ig_{\tau(i)}$ with $h_i\in H$. Then $\tau$ is a permutation of the set $\{1,\ldots,d\}$. The {\bf tensor induction} of $\rho$ maps $g$ to the element of $\GL(\otimes^{d}_AM)$ given by
$$
m_1\otimes\cdots\otimes m_d\mapsto \rho(h_1)m_{\tau(1)}\otimes\cdots\otimes\rho(h_d)m_{\tau(d)}.
$$
It is a representation $\otimes-\mathrm{Ind}(\rho)$ of $G$ whose isomorphism class is independent of the choice of $g_1,\ldots,g_d$. If $G$ is profinite and $\rho$ is continuous, so is $\otimes-\mathrm{Ind}(\rho)$.

Let $X$ be a connected scheme and $\pi:Y\to X$ a finite \'etale map of degree $d$. In \cite[10.5]{katz1990esa}, the tensor induction of representations is used to construct the tensor direct image of a lisse $\ell$-adic sheaf $\FF$ on $Y$: choosing geometric points $\bar y$ of $Y$ and $\bar x=\pi(\bar y)$ of $X$, and viewing $\FF$ as a finite dimensional $\ell$-adic representation of $\pi_1(Y,\bar y)$, which is an open subgroup of $\pi_1(X,\bar x)$ of index $d$, the \emph{tensor direct image} of $\FF$ is the lisse sheaf $\pi_{\otimes\ast}\FF$ on $X$ corresponding to the representation $\otimes-\mathrm{Ind}(\FF)$ of $\pi_1(X,x)$. It is well defined up to isomorphism.

A particularly important case is when $X$ is a variety over a finite field $k$ of characteristic $p\neq\ell$, $k_r$ is a finite extension of $k$ of degree $r$, and $Y=X\times\mathrm{Spec}(k_r)$ is the extension of scalars of $X$ to $k_r$. In that case, if $\FF$ is a lisse sheaf on $Y$, $t\in X(k)$ is a $k$-valued point of $X$ and $F_x\in\pi_1(X)$ denotes a geometric Frobenius element of $X$ at $t$, we have \cite[Proposition 2.1]{fu-wan-incomplete}
$$
\mathrm{Tr}(F_x|(\pi_{\otimes\ast}\FF)_{\bar x})=\mathrm{Tr}(F_x^r|\FF_{\bar x}).
$$
This is exploited in \cite{fu-wan-incomplete} to give estimates for some partial character sums.

Suppose that $H$ is normal in $G$. We will give a characterization of the tensor induction of a representation $\rho:H\to\GL(M)$ of $H$ that will be useful later to motivate the definition of the tensor direct image of a constructible sheaf. For every $g\in G$, let $g^\ast\rho:H\to\GL(M)$ be the representation given by
$$
(g^\ast\rho)(h)(m)=\rho(ghg^{-1})(m).
$$
It is clear that $(gg')^\ast\rho=g'^\ast g^\ast\rho$.

If $\phi:M\to M'$ is a homomorphism between the $H$-representations $\rho:H\to\GL(M)$ and $\rho':H\to\GL(M')$, then it is also a homomorphism between $g^\ast\rho$ and $g^\ast\rho'$, that we will denote by $g^\ast\phi$. So $\rho\mapsto g^\ast\rho$ defines an auto-equivalence of the category of representations of $H$ (over finitely generated free $A$-modules). If $g\in H$, then $\rho(g)$ is a isomorphism between $\rho$ and $g^\ast\rho$. More generally, if $g'g^{-1}\in H$, then $\rho(g'g^{-1})$ is an isomorphism bewteen $g^\ast\rho$ and $g'^\ast\rho$.

For every $\sigma\in G/H$, pick a representative $\tilde\sigma\in \sigma$, and let $h_{\sigma,\tau}=\tilde\sigma\tilde\tau\widetilde{\sigma\tau}^{-1}\in H$ for every $\sigma,\tau\in G/H$. We define the following category $\mathcal C_{G,H}$: the objects of $\mathcal C_{G,H}$ are representations $\rho:H\to\GL(M)$ together with isomorphisms $\Psi_\sigma:\rho\to \tilde\sigma^\ast\rho$ for each  $\sigma\in G/H$ satisfying the following cocycle condition for every $\sigma,\tau\in G/H$:
\begin{equation}\label{cocycle-group}
(\tilde\tau^\ast\Psi_\sigma)\circ \Psi_\tau=\rho(h_{\sigma,\tau})\circ\Psi_{\sigma\tau}:\rho\stackrel{\cong}{\to}\tilde\tau^\ast\tilde\sigma^\ast\rho
\end{equation}
A morphism in $\CCC_{G,H}$ is a homomorphism of representations $\Phi:\rho\to\rho'$ such that 
$$\Psi'_\sigma\circ\Phi=(\tilde\sigma^\ast\Phi)\circ\Psi_\sigma:\rho\to\tilde\sigma^\ast\rho'
$$
for every $\sigma\in G/H$.

Given a representation $\pi:G\to\GL(M)$, the restriction $\rho=\pi_{|H}$ of $\pi$ to $H$ together with the isomorphisms $\Psi_\sigma:\rho\to \tilde \sigma^\ast\rho$ given by $\Psi_\sigma(m)=\pi(\tilde\sigma)(m)$ is an object of $\mathcal C_{G,H}$, and every homomorphism of representations of $G$ $\pi\to\pi'$ induces a morphism in $\CCC_{G,H}$ between their restrictions. This defines a (covariant) functor $F$ from the category $\mathrm{Rep}_A(G)$ of representations of $G$ to ${\mathcal C}_{G,H}$.

\begin{proposition}\label{eqcat}
The functor $F$ is an equivalence of categories.
\end{proposition}

\begin{proof}
 $F$ is clearly flat: a morphism between two representations is zero if and only if it is zero when restricted to a given subgroup. Let us check that it is faithful: let $\Phi:\pi_H\to\pi'_H$ be a homomorphism in $\CCC_{G,H}$, we need to show that it is also a homomorphism of representations beween $\pi$ and $\pi'$.
 
  Let $g\in G$, and write $g=h \tilde\sigma$, where $\sigma$ is the class of $g$ in $G/H$ and $h\in H$. Then
  $$
  \Phi\circ\pi(g)=(\tilde\sigma^\ast\Phi)\circ\pi(g)=(\tilde\sigma^\ast\Phi)\circ\pi(\tilde\sigma h)=(\tilde\sigma^\ast\Phi)\circ\Psi_\sigma\circ\pi(h)=
  $$
  $$
  =\Psi'_\sigma\circ\Phi\circ\pi(h)=\Psi'_\sigma\circ\pi'(h)\circ\Phi=\pi'(\tilde\sigma h)\circ\Phi=\pi'(g)\circ\Phi.$$

 Finally, let us check that $F$ is essentially surjective. Given an object $(\rho,\{\Psi_\sigma\})$ of $\mathcal C_{G,H}$ and $g\in G$, write $g=h\tilde\sigma$ as above, and let $\pi(g):M\to M$ be the automorphism given by $\pi(g)=\rho(h)\circ\Psi_\sigma$. This defines a representation of $G$: if $g'=h'\tilde\sigma'$ with $h'\in H$, then $gg'=h \tilde\sigma h'\tilde\sigma' =(h\tilde\sigma h'\tilde\sigma^{-1})(\tilde\sigma\tilde\sigma')=(h\tilde\sigma h'\tilde\sigma^{-1}h_{\sigma,\sigma'})\widetilde{\sigma\sigma'}$, so
 $$
 \pi(gg')=\rho(h\tilde\sigma h'\tilde\sigma^{-1}h_{\sigma,\sigma'})\circ\Psi_{\sigma\sigma'}=\rho(h)\circ\rho(\tilde\sigma h'\tilde\sigma^{-1})\circ\rho(h_{\sigma,\sigma'})\circ\Psi_{\sigma\sigma'}=
 $$
 $$
 =\rho(h)\circ\rho(\tilde\sigma h'\tilde\sigma^{-1})\circ(\tilde\sigma'^\ast\Psi_\sigma)\circ \Psi_{\sigma'}=\rho(h)\circ\Psi_\sigma\circ\rho(h')\circ\Psi_{\sigma'}=\pi(g)\circ\pi(g')
 $$
 since $\Psi_\sigma\circ\rho(h')=(\tilde\sigma^\ast\rho)(h')\circ\Psi_\sigma=\rho(\tilde\sigma h'\tilde\sigma^{-1})\circ\Psi_\sigma=\rho(\tilde\sigma h'\tilde\sigma^{-1})\circ(\tilde\sigma'^\ast\Psi_\sigma)$, given that $\Psi_\sigma$ is an isomorphism between $\rho$ and $\tilde\sigma^\ast\rho$.
 
 This gives an object of $\mathrm{Rep}_A(G)$ whose associated object of $\CCC_{G,H}$ is $(\rho,\{\Psi_\sigma\})$.
\end{proof}

Using this equivalence of categories, we can characterize the tensor induction of $\rho:H\to\GL(M)$ as the representation of $G$ corresponding to the object of $\mathcal C_{G,H}$ given by the representation
$$
{\mathcal M}:=\bigotimes_{\tau\in G/H}\tilde\tau^\ast M
$$
of $H$ and the isomorphisms
$$
\Psi_\sigma:{\mathcal M}\to \tilde\sigma^\ast{\mathcal M}=\bigotimes_{\tau\in G/H} (\tilde\tau\tilde\sigma)^\ast M =\bigotimes_{\tau\in G/H} (h_{\tau,\sigma}\widetilde{\tau\sigma})^\ast M 
=\bigotimes_{\tau\in G/H} \widetilde{\tau\sigma}^\ast h_{\tau,\sigma}^\ast M
$$
given by 
$$
\bigotimes_{\tau\in G/H} m_\tau\mapsto\bigotimes_{\tau\in G/H}\rho(h_{\tau,\sigma})(m_{\tau\sigma}).
$$
An easy but tedious computation shows that this is indeed an object of $\mathcal C_{G,H}$.

\section{Tensor direct image of sheaves}

In this section we will give a geometric definition of the tensor direct image of a $\ell$-adic sheaf which extends the one given in \cite[10.5]{katz1990esa} for lisse sheaves to arbitrary constructible sheaves and even to the derived category.

Fix a prime $\ell$, let $E_\lambda$ be a finite extension of ${\mathbb Q}_\ell$, ${\mathfrak o}_\lambda$ its ring of integers, with maximal ideal ${\mathfrak m}_\lambda$, and let $A={\mathfrak o}_\lambda/{\mathfrak m}_\lambda^n$ for some $n\geq 1$. Let $\pi:Y\to X$ be a Galois finite \'etale morphism of schemes on which $\ell$ is invertible, with Galois group $G=\mathrm{Aut}(Y/X)$. By \cite[I.5.4]{milne1980ecv}, this means that the morphism
$$
\coprod_{\sigma\in G} Y_\sigma\to Y\times_X Y
$$
given by $y\in Y_\sigma\cong Y\mapsto (y,\sigma y)$ is an isomorphism.

Following the construction given in the previous section for group representations, let $\mathcal C_{X,Y}$ be the following category: the objects in $\mathcal C_{X,Y}$ consist of a constructible \'etale sheaf of free $A$-modules $\FF$ on $Y$, together with isomorphisms $\Psi_\sigma:\FF\to\sigma^\ast\FF$ for every $\sigma\in G$ satisfying the cocycle condition:
$$
(\tau^\ast\Psi_\sigma)\circ\Psi_\tau=\Phi_{\sigma,\tau}\circ\Psi_{\sigma\tau}:\FF\to\tau^\ast\sigma^\ast\FF
$$
for every $\sigma,\tau\in G$, where $\Phi_{\sigma,\tau}:(\sigma\tau)^\ast\FF\to\tau^\ast\sigma^\ast\FF$ is the isomorphism induced by the natural transformation $(\sigma\tau)^\ast\to\tau^\ast\sigma^\ast$ (note that this implies that $\Psi_{\mathrm{Id}_Y}=\mathrm{Id}_\FF$). A morphism between two objects $(\FF,\{\Psi_\sigma\}_{\sigma\in G})$ and $(\FF',\{\Psi'_\sigma\}_{\sigma\in G})$ is a morphism $\rho:\FF\to\FF'$ such that $\Psi'_\sigma\circ\rho=(\sigma^\ast\rho)\circ\Psi_\sigma$ for every $\sigma\in G$.

For every \'etale sheaf of free $A$-modules $\GGG$ on $X$, the sheaf $\FF:=\pi^\ast\GGG$ together with the isomorphisms $\Psi_\sigma:\FF\to\sigma^\ast\FF$ induced by the natural transformations $\pi^\ast=(\pi\sigma)^\ast\to\sigma^\ast\pi^\ast$ is an object of $\mathcal C_{X,Y}$, and for every morphism $\rho:\GGG\to\GGG'$ of sheaves on $X$ the restriction $\pi^\ast\rho:\pi^\ast\GGG\to\pi^\ast\GGG'$ is clearly a morphism in $\CCC_{X,Y}$. This defines a functor $F$ from the category of \'etale sheaves of free $A$-modules on $X$ to $\mathcal C_{X,Y}$. The analogue to Proposition \ref{eqcat} is the following

\begin{proposition}\label{eqcat2} $F$ is an equivalence of categories.
\end{proposition}

\begin{proof}

Consider the diagram
$$
\begin{tikzcd}
Y\times_X Y\times_X Y\arrow[r, shift left=3.5, "p_{12}"] \arrow[r, "p_{13}"] \arrow[r, shift right=3.5, "p_{23}"] & Y\times_X Y \arrow[r, shift left=2, "p_1"] \arrow[r, shift right=2, "p_2"] &  Y \arrow[r, "\pi"] & X
\end{tikzcd}
$$
By descent theory (see eg. \cite[Lemma 7.25.5]{stacks}), the pull-back $\FF\to\pi^\ast\FF$ is an equivalence of categories between the category of étale sheaves of free $A$-modules on $X$ and the category of descent data of \'etale sheaves of free $A$-modules for the \'etale cover $\pi:Y\to X$ (that is, an étale sheaf $\FF$ of free $A$-modules on $Y$ endowed with an isomorphism $\Psi:p_1^\ast\FF\to p_2^\ast\FF$ such that $p_{13}^\ast\Psi=(p_{23}^\ast\Psi)\circ(p_{12}^\ast\Psi)$ modulo the natural transformations $p_{13}^\ast p_1^\ast\cong p_{12}^\ast p_1^\ast$ and $p_{13}^\ast p_2^\ast\cong p_{23}^\ast p_2^\ast$).

Remember that there is a natural isomorphism
$$
\coprod_{\sigma\in G} Y_\sigma\stackrel{\cong}{\to} Y\times_X Y
$$
where $Y_\sigma\cong Y$, given by $y\in Y_\sigma\mapsto (y,\sigma y)$. Under this isomorphism, $p_1^\ast\FF$ and $p_2^\ast\FF$ correspond to $\FF$ and $\sigma^\ast\FF$ on $Y_\sigma$, so giving an isomorphism $p_1^\ast\FF\to p_2^\ast\FF$ is equivalent to giving an isomorphism $\FF\to\sigma^\ast\FF$ for every $\sigma\in G$. So every object $(\FF,\{\Psi_\sigma\})$ of $\mathcal C_{X,Y}$ determines an isomorphism $\Psi:p_1^\ast\FF\to p_2^\ast\FF$.

Let us check that this isomorphism satisfies the cocycle condition. Notice that
$$
(Y\times_X Y)\times_X Y\cong (\coprod_{\sigma\in G} Y_\sigma)\times_X Y\cong\coprod_{\sigma\in G}(Y_\sigma\times_X Y)\cong\coprod_{\sigma,\tau\in G}Y_{(\sigma,\tau)}
$$
where, under this isomorphism, $y \in Y_{(\sigma,\tau)}$ corresponds to $(y,\sigma y,\tau y)\in Y\times_X Y\times_X Y$. In particular, $p_{12},p_{13}$ and $p_{23}$, viewed as maps $\coprod_{\sigma,\tau\in G} Y_{(\sigma,\tau)}\to\coprod_{\sigma\in G} Y_\sigma$, take $y \in Y_{(\sigma,\tau)}$ to $y\in Y_\sigma$, $y\in Y_\tau$ and $\sigma y\in Y_{\tau\sigma^{-1}}$ respectively, and the natural transformations $p_{13}^\ast p_1^\ast\cong p_{12}^\ast p_1^\ast$ and $p_{13}^\ast p_2^\ast\cong p_{23}^\ast p_2^\ast$ correspond to the identity and the natural transformation $\tau^\ast\cong\sigma^\ast(\tau\sigma^{-1})^\ast$ respectively. Then (after a ``change of variables'' $(\sigma,\tau)\mapsto(\tau,\sigma\tau)$):
$$(p_{12}^\ast\Psi)_{|Y_{(\tau,\sigma\tau)}}=\Psi_\tau:\FF\to\tau^\ast \FF,$$
$$(p_{13}^\ast\Psi)_{|Y_{(\tau,\sigma\tau)}}=\Psi_{\sigma\tau}:\FF\to(\sigma\tau)^\ast \FF \mbox{ and}$$
$$
(p_{23}^\ast\Psi)_{|Y_{(\tau,\sigma\tau)}}=\tau^\ast\Psi_{\sigma}:\tau^\ast\FF\to\tau^\ast\sigma^\ast\FF
, \mbox{ so}$$
$(p_{23}^\ast\phi)\circ(p_{12}^\ast\phi)$, in the connected component $Y_{(\tau,\sigma\tau)}$, translates to $\Phi_{\sigma,\tau}\circ\Psi_{\sigma\tau}=(\tau^\ast\Psi_\sigma)\circ\Psi_\tau$, which holds for all objects of $\mathcal C_{X,Y}$ by definition.

This defines an equivalence of categories between $\mathcal C_{X,Y}$ and the category of descent data for $\pi:Y\to X$, which is equivalent via $\pi^\ast$ to the category of sheaves on $X$.
\end{proof}

It is clear that, under this correspondence, locally constant objects of $\mathcal C_{X,Y}$ (i.e. objects whose sheaf is locally constant) correspond to locally constant sheaves on $X$. By passing to the limit, this equivalence of categories also applies to sheaves with coefficients on ${\mathfrak o}_\lambda$ and on $E_\lambda$.

The equivalence of categories is compatible with pull-backs in the following sense: Let $g:Z\to X$ be a morpism of schemes and $\tilde \pi:Z\times_X Y\to Z$ the base change of $\pi$ with respect to $g$. Then $\tilde \pi$ is also a Galois \'etale morphism with Galois group $G$, with the elements of $G$ acting on $Z\times_X Y$ via the second factor. Let $({\mathcal F},\{\Psi_\sigma\}_{\sigma\in G})$ be an object of $\CCC_{X,Y}$. Taking pull-backs with respect the projection map $\tilde g:Z\times_X Y\to Y$ we get a sheaf $\tilde g^\ast\FF$ on $Z\times_X Y$ and isomorphisms $\tilde g^\ast\Psi_\sigma:\tilde g^\ast\FF\to\tilde g^\ast\sigma^\ast\FF$ for every $\sigma\in G$. By composition with the natural transformations $\tilde g^\ast\sigma^\ast\cong(\sigma\circ\tilde g)^\ast= (\tilde g\circ(1\times\sigma))^\ast\cong(1\times\sigma)^\ast \tilde g^\ast$ we get isomorphisms $\tilde\Psi_\sigma:\tilde g^\ast\FF\to(1\times\sigma)^\ast(\tilde g^\ast\FF)$ for every $\sigma \in G$.

\begin{proposition}\label{base-change} $(\tilde g^\ast\FF,\{\tilde\Psi_\sigma\}_{\sigma\in G})$ is an object of $\CCC_{Z,Z\times_X Y}$. If $({\mathcal F},\{\Psi_\sigma\})$ corresponds to the sheaf $\GGG$ on $X$, $(\tilde g^\ast\FF,\{\tilde\Psi_\sigma\})$ corresponds to $g^\ast\GGG$ on $Z$.
\end{proposition}

\begin{proof}
 Given the equivalence of categories above, we can assume that $({\mathcal F},\{\Psi_\sigma\})$ is associated to a sheaf $\GGG$ on $X$, that is, $\FF=\pi^\ast\GGG$ and $\Psi_\sigma:\FF\to\sigma^\ast\FF$ is given by the natural transformation $\pi^\ast=(\pi\sigma)^\ast\cong\sigma^\ast \pi^\ast$ for every $\sigma\in G$. Then its pull-back is given by $\tilde g^\ast\FF=\tilde g^\ast \pi^\ast\GGG$, with $\tilde\Psi_\sigma:\tilde g^\ast \pi^\ast\GGG\to(1\times\sigma)^\ast \tilde g^\ast \pi^\ast\GGG$ the isomorphism induced by the natural transformation $\tilde g^\ast \pi^\ast =\tilde g^\ast(\pi\sigma)^\ast \cong \tilde g^\ast(\sigma^\ast \pi^\ast)= (\tilde g^\ast \sigma^\ast) \pi^\ast\cong((1\times\sigma)^\ast\tilde g^\ast)\pi^\ast=(1\times\sigma)^\ast(\tilde g^\ast \pi^\ast)$.
 
 Composing with the natural transformation $\tilde \pi^\ast g^\ast\cong(g\tilde \pi)^\ast=(\pi\tilde g)^\ast\cong\tilde g^\ast \pi^\ast$, we observe that this pull-back is isomorphic to the object of $\CCC_{Z,Z\times_X Y}$ given by $\tilde \pi^\ast g^\ast\GGG$ with the isomorphisms $\tilde \Psi'_\sigma:\tilde \pi^\ast g^\ast\GGG\to(1\times\sigma)^\ast\tilde \pi^\ast g^\ast\GGG$ induced by the natural transformation $\tilde \pi^\ast=(\tilde \pi(1\times\sigma))^\ast\cong(1\times\sigma)^\ast\tilde \pi^\ast$, which is precisely the object of $\CCC_{Z,Z\times_X Y}$ associated to the sheaf $g^\ast\GGG$.
\end{proof}

Suppose now that $Y$ (and therefore $X$) is connected. Pick a geometric point $\bar y \hookrightarrow Y$, and let $\bar x= f(\bar y)\hookrightarrow X$. Let $\pi_1(Y,\bar y)$ and $\pi_1(X,\bar x)$ be the fundamental groups of $Y$ and $X$ with base points $\bar y$ and $\bar x$ respectively. We have an exact sequence
$$
1\to \pi_1(Y,\bar y) \to \pi_1(X,\bar x)\to G\to 1.
$$
There is an equivalence between the categories of locally constant sheaves of free $A$-modules on $X$ (respectively $Y$) and of continuous representations of $\pi_1(X,\bar x)$ (resp. of $\pi_1(Y,\bar y)$) on finite dimensional free $A$-modules \cite[V.1.2]{milne1980ecv}.

For every $\sigma\in G$, pick an element $\tilde\sigma\in \pi_1(X,\bar x)$ that lifts $\sigma$ and, for $\sigma,\tau\in G$, denote by $h_{\sigma,\tau}$ the element $\tilde\sigma\cdot\tilde\tau\cdot\widetilde{\sigma\tau}^{-1}\in \pi_1(Y,\bar y)$. If a sheaf $\FF$ on $Y$ corresponds to the representation $\rho$ of $Y$, then $\sigma^\ast\FF$ corresponds to $\tilde\sigma^\ast\rho$ (given by $(\tilde\sigma^\ast\rho)(h)=\rho(\tilde\sigma h\tilde\sigma^{-1})$). The natural transformation $(\sigma\tau)^\ast\to\tau^\ast\sigma^\ast$ corresponds to the natural transformation $\widetilde{\sigma\tau}^\ast\to\tilde\tau^\ast\tilde\sigma^\ast$ in the category of representations of $\pi_1(Y,\bar y)$ given, for every $\rho$, by the isomorphism of representations $\widetilde{\sigma\tau}^\ast\rho\to\tilde\tau^\ast\tilde\sigma^\ast \rho$ defined by $m\mapsto \rho(h_{\sigma,\tau})(m)$.

Let $(\FF,\{\Psi_\sigma\}_{\sigma\in G})$ be an object of $\mathcal C_{X,Y}$ with $\FF$ locally constant, $\rho$ the representation of $\pi_1(Y,\bar y)$ corresponding to $\FF$, and $\tilde\Psi_\sigma:\rho\to\tilde\sigma^\ast\rho$ the isomorphism corresponding to $\Psi_\sigma$. Then the paragraph above and the cocycle equation (\ref{cocycle-group}) imply that $(\rho,\{\tilde\Psi_\sigma\}_{\sigma\in G})$ is an object of $\mathcal C_{\pi_1(X,\bar x),\pi_1(Y,\bar y)}$. If $\FF=\pi^\ast\GGG$ for a locally constant $\GGG$ on $X$, then $\rho$ is the restriction to $\pi_1(Y,\bar y)$ of the representation $\tilde\rho$ of $\pi_1(X,\bar x)$ associated to $\GGG$, and the natural isomorphism $\FF\to\sigma^\ast\FF$ translates to the isomorphism of representations $\rho\to\tilde\sigma^\ast\rho$ given by $m\mapsto\tilde\rho(\tilde\sigma)(m)$. To sum up:

\begin{proposition}\label{correspondence}
If $(\FF,\{\Psi_\sigma\}_{\sigma\in G})$ is an object ot $\mathcal C_{X,Y}$ corresponding to the locally constant sheaf $\GGG$ on $Y$, then the object $(\rho,\{\tilde\Psi_\sigma\}_{\sigma\in G})$ of $\mathcal C_{\pi_1(X,\bar x),\pi_1(Y,\bar y)}$ corresponds to the representation $\tilde\rho$ of $\pi_1(X,\bar x)$ defined by $\GGG$.
\end{proposition}

Now let $\FF$ be a constructible sheaf of free $A$-modules on $Y$, where $A$ is either $E_\lambda$, ${\mathfrak o}_\lambda$ or ${\mathfrak o}_\lambda/{\mathfrak m}_\lambda^n$ for some $n\geq 1$. Take
$$
\GGG=\bigotimes_{\tau\in G}\tau^\ast\FF
$$
and, for every $\sigma\in G$, let
$$
\Psi_\sigma:\GGG=\bigotimes_{\tau\in G}\tau^\ast\FF\to\bigotimes_{\tau\in G}(\tau\sigma)^\ast\FF\cong \bigotimes_{\tau\in G}\sigma^\ast\tau^\ast\FF\cong\sigma^\ast\GGG
$$
be the natural permutation isomorphism $\otimes_{\tau} m_\tau\mapsto \otimes_{\tau} m_{\tau\sigma}$, composed with the isomorphisms induced by the natural transformations $(\tau\sigma)^\ast\to\sigma^\ast\tau^\ast$. These isomorphisms satisfy the cocycle condition, so they define an object of $\mathcal C_{X,Y}$, which corresponds to a constructible sheaf of $A$-modules on $X$.

\begin{defn}
 The {\bf tensor direct image} of $\FF$ by $\pi$ is the sheaf $\pi_{\otimes\ast}\FF$ on $X$ corresponding to this object of $\mathcal C_{X,Y}$.
\end{defn}

We now enumerate some of the basic properties of this construction

\begin{proposition}
 Let $\FF$ and $\FF'$ be constructible sheaves on $Y$. 

\begin{enumerate}\label{properties}
 \item $\pi^\ast(\pi_{\otimes\ast}\FF)\cong\bigotimes_{\tau\in G}\tau^\ast\FF$
 \item $\pi_{\otimes\ast}(\FF\otimes\FF')\cong(\pi_{\otimes\ast}\FF)\otimes(\pi_{\otimes\ast}\FF')$
 \item (base change) Let $g:Z\to X$ be a morphism of schemes, and
 $$
\begin{tikzcd}
 Z\times_X Y \arrow[r, "\tilde g"] \arrow[d, "\tilde\pi"] & Y \arrow[d,"\pi"] \\
  Z \arrow[r,"g"] & X 
\end{tikzcd}
$$
 the pull-back. Then $g^\ast(\pi_{\otimes\ast}\FF)\cong\tilde \pi_{\otimes\ast} \tilde g^\ast\FF$.
\end{enumerate}
 
\end{proposition}

\begin{proof}
 The first property is clear by construction. For the second one, note that there is a natural isomorphism
 $$
 \bigotimes_{\tau\in G}\tau^\ast(\FF\otimes\FF')\cong
 \bigotimes_{\tau\in G}(\tau^\ast\FF)\otimes(\tau^\ast\FF')\cong
 (\bigotimes_{\tau\in G}\tau^\ast\FF)\otimes(\bigotimes_{\tau\in G}\tau^\ast\FF')
 $$
 which commutes with the isomorphisms $\Psi_\sigma$.
 
 The third property is a particular case of Proposition \ref{base-change}, via the natural isomorphism
 $$
 \tilde g^\ast\left(\bigotimes_{\tau\in G}\tau^\ast\FF\right)\cong \bigotimes_{\tau\in G}\tilde g^\ast\tau^\ast\FF\cong \bigotimes_{\tau\in G}(1\times\tau)^\ast\tilde g^\ast\FF
 $$
\end{proof}

The next result shows that this operation does in fact generalize the tensor direct image of locally constant sheaves defined in \cite[10.5]{katz1990esa} to arbitrary (constructible) sheaves.

\begin{proposition}
 Suppose that $Y$ is connected, and let $\FF$ be a locally constant sheaf on $Y$, corresponding to the representation $\rho$ of $\pi_1(Y,\bar y)$. Then $\pi_{\otimes\ast}\FF$ is a locally contant sheaf on $X$ corresponding to the representation $\otimes-\mathrm{Ind}(\rho)$ of $\pi_1(X,\bar x)$.
\end{proposition}

\begin{proof}
 By Proposition \ref{correspondence} and the construction of $\pi_{\otimes\ast}\FF$, the representation associated to $\pi_{\otimes\ast}\FF$ is the representation corresponding to the object of $\CCC_{\pi_1(X,\bar x),\pi_1(Y,\bar y)}$ given by $\bigotimes_{\tau\in G}\tilde\tau^\ast\rho$, with the isomorphisms
 $$
\tilde\Psi_\sigma:\bigotimes_{\tau\in G}\tilde\tau^\ast\rho\to
\bigotimes_{\tau\in G}\tilde\sigma^\ast\tilde\tau^\ast\rho
 $$
 being the composition of the permutation map $\bigotimes_{\tau\in G}\tilde\tau^\ast\rho\to\bigotimes_{\tau\in G}\widetilde{\tau\sigma}^\ast\rho$ and the tensor product of the natural isomorphisms $\widetilde{\tau\sigma}^\ast\rho\to\tilde\sigma^\ast\tilde\tau^\ast\rho$ given by $m\mapsto\rho(h_{\tau,\sigma})(m)$. But this is precisely the tensor induction of $\rho$, according to the characterization given in the previous section.
\end{proof}

We will be mainly interested in the following situation: $X$ is an algebraic variety over the finite field $k$, $k_r$ is a finite extension of $k$ of degree $r$ and $Y=X\times\Spec\; k_r$ is the extension of scalars of $X$. Then the projection $\pi:Y\to X$ induced by the inclusion $k\hookrightarrow k_r$ is a Galois \'etale map, whose Galois group is cyclic and generated by the Frobenius automorphism $\sigma$ of $\Spec\;k_r/\Spec\;k$. Here
$$
\GGG=\FF\otimes\sigma^\ast\FF\otimes\cdots\otimes\sigma^{(r-1)\ast}\FF
$$
and
$$\Psi_\sigma:\GGG\to\sigma^\ast\GGG=\sigma^\ast\FF\otimes\sigma^{2\ast}\FF\otimes\cdots\sigma^{r\ast}\FF=\sigma^\ast\FF\otimes\sigma^{2\ast}\FF\otimes\cdots\otimes\FF
$$
is given by $m_0\otimes m_1\otimes\cdots\otimes m_{r-1}\mapsto m_1\otimes m_2\otimes\cdots\otimes m_0$.

Let $x:\Spec\;k\to X$ be a $k$-valued point of $X$, $x_r:\Spec\;k_r\to\Spec\;k\to X$ the corresponding $k_r$-valued point and $\bar{x_r}=\bar x:\Spec\;\bar k\to X$ a geometric point over $x$. Let $F_x\in\pi_1(X,\bar x)$ be a geometric Frobenius element at $x$, then $F_x^r$ is in $\pi_1(Y,\bar x)$ (since the quotient is cyclic of degree $r$) and is a geometric Frobenius element of $Y$ at $\bar x_r$. The following result generalizes \cite[Proposition 2.1]{fu-wan-incomplete}

\begin{proposition}\label{frobenius}
 With the previous notation, we have
 $$
 \Tr(F_x|(\pi_{\otimes\ast}\FF)_{\bar x})=\Tr(F_x^r|\FF_{\bar x_r})
 $$
\end{proposition}

\begin{proof} We have a cartesian diagram
$$
\begin{tikzcd}
\Spec\;k_r \arrow[r,"x_r"] \arrow[d,"\tilde\pi"] & Y \arrow[d,"\pi"] \\
\Spec\;k \arrow[r,"x"] & X
\end{tikzcd}
$$
where $\tilde \pi:\Spec\;k_r\to\Spec\;k$ is the canonical projection.

By Proposition \ref{properties}(3), we get an isomorphism of sheaves on $\Spec\;k$ (i.e. of $\Gal(\bar k/k)$-modules) $(\pi_{\otimes\ast}\FF)_{\bar x}\cong \tilde \pi_{\otimes\ast}(\FF_{\bar x_r})$. Since every sheaf on $\Spec\;k$ is locally constant (and therefore corresponds to a representation of $\pi_1(\Spec\;k,\Spec\;\bar k)=\Gal(\bar k/k)$), as a representation of $\Gal(\bar k/k)$ $(\pi_{\otimes\ast}\FF)_{\bar x}$ is the tensor induction of the representation $\FF_{\bar x_r}$ of $\Gal(\bar k/k_r)$. By \cite[Proposition 2.1]{fu-wan-incomplete} we conclude that $\Tr(F_x|(\pi_{\otimes\ast}\FF)_{\bar x})=\Tr(F_x^r|\FF_{\bar x_r})$.
\end{proof}

\section{Tensor direct image in the derived category}

In this section we keep the same geometric setup as in the previous one: $\pi:Y\to X$ is a Galois finite étale morphism with Galois group $G$, defined over a base ring good enough for the derived category of étale sheaves to be well behaved (e.g. a field).

 Let $\ell$ be a prime invertible in $X$, $E_\lambda$, ${\mathfrak o}_\lambda$ and ${\mathfrak m}_\lambda$ as above and let $A={\mathfrak o}_{\lambda,n}={\mathfrak o}_\lambda/{\mathfrak m}_\lambda^n$ for some $n\geq 1$. Let ${\mathcal E}_X$ be the category of bounded complexes of (constructible) sheaves of $A$-modules on $X$. We have an obvious equivalence of categories between the category of bounded complexes of objects in ${\mathcal C}_{X,Y}$ and the category ${\mathcal E}_{X,Y}$ whose objects are bounded complexes of sheaves $P$ on $Y$ together with isomorphisms $P\to\sigma^\ast P$ for every $\sigma\in G$ satisfying the cocycle conditions. Therefore, from Proposition \ref{eqcat2} we deduce
 
 \begin{proposition}\label{eqcat3}
  The functor ${\mathcal E}_X\to{\mathcal E}_{X,Y}$ given by $Q\mapsto \pi^\ast Q$ is an equivalence of categories.
 \end{proposition}

By taking projective limits, the same statement is true if we replace $A$ with the full ring of integers ${\mathfrak o}_\lambda$ of $E_\lambda$ and, by tensoring with $\mathbb Q$, with $E_\lambda$ itself.


\begin{defn}
 Let $P\in{\mathcal E}_Y$ (where $A=E_\lambda$, ${\mathfrak o}_\lambda$ or ${\mathfrak o}_{\lambda,n}$ for some $n\geq 1$), the {\bf tensor direct image} of $P$ is the object $\pi_{\otimes\ast}P\in {\mathcal E}_X$ corresponding, under the previous equivalence of categories, to the object
 $$
 \bigotimes_{\tau\in G} \tau^\ast P\in\mathcal E_Y
 $$
 and the natural commutation isomorphisms
 $$
 \Psi_\sigma:\bigotimes_{\tau\in G} \tau^\ast P\to\bigotimes_{\tau\in G} (\tau\sigma)^\ast P\cong\bigotimes_{\tau\in G} \sigma^\ast\tau^\ast P\cong\sigma^\ast\left(\bigotimes_{\tau\in G} \tau^\ast P\right).
 $$
\end{defn}

If $P=\FF[0]$ for a single sheaf $\FF$, then $\pi_{\otimes\ast}P=\pi_{\otimes\ast}\FF[0]$ by construction. The basic properties stated in Proposition \ref{properties} are still valid for complexes, with a similar proof.

If $f:P\to P'$ is a quasi-isomorphism in ${\mathcal E}_{X}$, then so is the induced morphism $\pi_{\otimes\ast}f:\pi_{\otimes\ast}P\to\pi_{\otimes\ast}P'$, so we can unambiguously speak about the tensor direct image of an object of the derived category $D_{ctf}^b(Y,A)$ of $ctf$-complexes of sheaves of $A$-modules on $X$ (the full subcategory of the derived category consisting of objects which are quasi-isomorphic to a bounded complex of $A$-flat sheaves \cite[II.5]{kiehl-weissauer}) as an object of $D_{ctf}^b(X,A)$.

Let us now focus on the setup where $X$ is an algebraic variety over the finite field $k$, $k_r$ is a finite extension of $k$ of degree $r$ and $\pi:Y=X\times\Spec\; k_r\to X$ is the projection. Let $F_x\in\pi_1(X,\bar x)$ be a geometric Frobenius element at $x$, then $F_x^r\in\pi_1(Y,\bar x)$ is a geometric Frobenius element of $Y$ at $\bar x_r$. Given an element $P^\bullet\in D^b_{ctf}(Y,A)$, there is a well-defined value of $\Tr(F_x|P^\bullet_{\bar x})$, given as the alternating sum $\sum_i(-1)^i\Tr(F_x|P^i_{\bar x})$ if $P^\bullet$ is a $ctf$ complex (which is well defined since it is invariant under quasi-isomorphism). Then we have the following result similar to Proposition \ref{frobenius}:

\begin{proposition}\label{tracetensorind}
 For any $P^\bullet\in\Dbc(Y,A)$ we have
 $$
 \Tr(F_x|(\pi_{\otimes\ast}P^\bullet)_{\bar x})=\Tr(F_x^r|P^\bullet_{\bar x_r})
 $$
\end{proposition}

\begin{proof}
 By the usual passage to the limit and torsion killing process, it suffices to prove it when $A={\mathfrak o}_{\lambda,n}$. By base change, we may assume that $X=\Spec\;k$ and $Y=\Spec\;k_r$. In that case we can view $P^\bullet$ as a bounded complex of representations of $\Gal(\bar k/k_r)$ on flat (or, equivalently, free, see \cite[5.1]{kiehl-weissauer}) $A$-modules. Since $\{F_x^i\}_{i=0}^{r-1}$ is a set of representatives of the cosets of $\Gal(\bar k/k_r)$ in $\Gal(\bar k/k)$, $\pi_{\otimes\ast}P^\bullet$, as a representation of $\Gal(\bar k/k_r)$, is given by
 $$
 Q^\bullet=P^\bullet\otimes F_x^\ast P^\bullet\otimes \cdots \otimes F_x^{(r-1)\ast}P^\bullet
 $$
 with the isomorphism
 $$
 \Psi_{F_x}:Q^\bullet=P^\bullet\otimes F_x^\ast P^\bullet\otimes \cdots \otimes F_x^{(r-1)\ast}P^\bullet\to F_x^\ast Q^\bullet= F_x^\ast P^\bullet\otimes F_x^{2\ast} P^\bullet\otimes \cdots \otimes F_x^{r\ast}P^\bullet
 $$
 given by the commutation map $P^\bullet\otimes F_x^\ast P^\bullet\otimes \cdots \otimes F_x^{(r-1)\ast}P^\bullet\to F_x^\ast P^\bullet\otimes F_x^{2\ast} P^\bullet\otimes \cdots \otimes P^\bullet$ followed by the isomorphism induced by the action of $F_x^r: P^\bullet\to F_x^{r\ast}P^\bullet$.
 
 Let $m\in{\mathbb Z}$; then, as $A$-modules,
 $$Q^m=\bigoplus_{i_0+\ldots+i_{r-1}=m} P^{i_0}\otimes  P^{i_1}\otimes\cdots\otimes P^{i_{r-1}}
 $$
 with $F_x$ acting on it by
 $$
 m_0\otimes m_1\otimes\cdots\otimes m_{r-1}\in P^{i_0}\otimes P^{i_1}\otimes\cdots\otimes  P^{i_{r-1}}\mapsto$$
 $$
 \mapsto (-1)^{i_0(i_1+\cdots+i_{r-1})}m_1\otimes m_2\otimes\cdots\otimes m_{r-1}\otimes F_x^r(m_0)\in  P^{i_1}\otimes  P^{i_2}\otimes\cdots\otimes P^{i_0}.
 $$
  In particular, if we fix bases ${\mathcal B}_j$ of $P^{i_j}$ and $m_j\in{\mathcal B}_j$ for every $j=0,\ldots,r-1$, then in the expression of $F_x\cdot(m_0\otimes m_1\otimes\cdots\otimes m_{r-1})$ with respect to the tensor product of the ${\mathcal B}_j$'s, the element $m_0\otimes m_1\otimes\cdots\otimes m_{r-1}$ appears with non-zero coefficient if and only if $m_j=m_{j+1}$ for all $j=0,\ldots,r-2$ (so, in particular, $i_j=i_{j+1}$ and $m=i_0r$) and $m_0$ appears with non-zero coefficient in the expansion of $F_x^r\cdot m_0$ with respect to ${\mathcal B}_0$. Additionally, in that case, these two non-zero coefficients concide up to multiplication by $(-1)^{i_0^2(r-1)}=(-1)^{i_0(r-1)}$. Putting all this together, we get:
 $$
 \Tr(F_x|Q^m)=\left\{\begin{array}{ll}
                      0 & \mbox{if }r\not|m \\
                      (-1)^{i(r-1)}\Tr(F_x^r|P^i) & \mbox{if }m=ir
                     \end{array}\right.
 $$
 so
 $$
 \Tr(F_x|\pi_{\otimes\ast}P^\bullet)=\sum_m(-1)^m\Tr(F_x|Q^m)=
 $$
 $$
 =\sum_i(-1)^{ir}(-1)^{i(r-1)}\Tr(F_x^r|P^i)=\sum_i(-1)^i\Tr(F_x^r|P^i)=\Tr(F_x^r|P^\bullet)
 $$
\end{proof}

\section{Convolution direct image of perverse sheaves}

Let $X$, $Y$, $\ell$ and $A$ be as in the previous section. Let ${\mathcal D}_X=D_{ctf}^b(X,A)$ (respectively ${\mathcal D}_Y=D_{ctf}^b(Y,A)$) be the derived category of $ctf$-complexes of sheaves of $A$-modules on $X$ (resp. on $Y$). Let $\mathcal{D}_{X,Y}$ be the following category: the objects of $\mathcal{D}_{X,Y}$ consist of an object $P\in{\mathcal D}_Y$ together with isomorphisms $\Psi_\sigma:P\to\sigma^\ast P$ for every $\sigma\in G$ satisfying the cocycle condition
$$
(\tau^\ast\Psi_\sigma)\circ\Psi_\tau=\Phi_{\sigma,\tau}\circ\Psi_{\sigma\tau}:P\to\tau^\ast\sigma^\ast P
$$
for every $\sigma,\tau\in G$, where $\Phi_{\sigma,\tau}:(\sigma\tau)^\ast P\to\tau^\ast\sigma^\ast P$ is the isomorphism defined by the natural transformation $(\sigma\tau)^\ast\to\tau^\ast\sigma^\ast$. A morphism between two objects $(P,\{\Psi_\sigma\}_{\sigma\in G})$ and $(P',\{\Psi'_\sigma\}_{\sigma\in G})$ is a morphism $\rho:P \to P'$ such that $\Psi'_\sigma\circ\rho=(\sigma^\ast\rho)\circ\Psi_\sigma$ for every $\sigma\in G$.

We have a natural functor $\mathcal D_X\to\mathcal D_{X,Y}$ given by $Q\mapsto \pi^\ast Q$. Since the morphisms in the derived category are defined up to homotopy, we can not expect this to be an equivalence of categories. However, let $\mathcal Perv_{X,Y}$ denote the full subcategory of $\mathcal D_{X,Y}$ consisting of objects $(P,\{\Psi_\sigma\}_{\sigma\in G})$ such that $P$ is perverse, and $\mathcal Perv_X$ the full subcategory of perverse objects in $\mathcal D_X$. Since $\pi^\ast$ is exact for the perverse $t$-structure, we have a functor $F:\mathcal Perv_X\to\mathcal Perv_{X,Y}$ given by $Q\mapsto \pi^\ast Q$.

\begin{proposition}\label{equivperv} The functor $F:\mathcal Perv_X\to\mathcal Perv_{X,Y}$ is an equivalence of categories. 
\end{proposition}

\begin{proof}
 Let $P,Q\in\mathcal Perv_X$. For every \'etale map $\varpi:U\to X$, $\varpi^\ast P$ and $\varpi^\ast Q$ are perverse objects on $U$, and in particular $\mathrm{Ext}^i_U(\varpi^\ast P,\varpi^\ast Q)=\mathrm{Hom}_U(\varpi^\ast P,\varpi^\ast Q[i])=0$ for every $i<0$. That is, the sheaf $\mathcal Ext^i(P,Q)$ vanishes on the \'etale site over $X$ for $i<0$.
 
 By \cite[Proposition 3.2.2]{beilinson1982faisceaux}, $U\mapsto\mathrm{Hom}_{D^b_c(U)}(\varpi^\ast P,\varpi^\ast Q)$ is a sheaf. In particular, applying it to the \'etale cover of $X$ given by $\{(Y,\pi\sigma)\}_{\sigma\in G}$ shows that $F$ is fully faithful.
 
 Additionally, taking $P=Q$ and applying \cite[Th\'eor\`eme 3.2.4]{beilinson1982faisceaux} to the same \'etale cover, we conclude that $F$ is essentially surjective.
\end{proof}

From now on we assume that $Y$ and $X$ are commutative group schemes of dimension $d$ over a base field $k$, and $\pi:Y\to X$ a finite Galois étale map as before. We have an exact convolution bi-functor $-\ast_!-$ from $\Dbc(Y,A)\times\Dbc(Y,A)$ to $\Dbc(Y,A)$ given by
$$
K\ast_! L=\R\mu_!(K\boxtimes L)[d]
$$
where $\mu:Y\times_k Y\to Y$ is the multiplication map. See \cite[8.1]{katz1990esa} for an exhaustive account of its basic properties. For our purposes, the most important ones are the existence of natural commutativity
$$
K\ast_!L\to L\ast_!K
$$
and associativity
$$
(K\ast_!L)\ast_!M\to K\ast_!(L\ast_!M)
$$
isomorphisms from which one can construct natural multi-commutativity isomorphisms
$$
K_1\ast_!K_2\ast_!\cdots\ast_!K_r\to K_{\sigma(1)}\ast_!K_{\sigma(2)}\ast_!\cdots\ast_!K_{\sigma(r)}
$$
for every $\sigma\in{\mathfrak S}_r$.

Let $P\in\mathcal Perv_Y$. $P$ is said to have \emph{property $\mathcal P_!$} \cite[2.6]{katz1996rls} if $P\ast_!Q\in\mathcal Perv_Y$ for every $Q\in\mathcal Perv_Y$. Suppose that $P$ has property $\mathcal P_!$. Given the commutative and associative properties of the convolution, there is a well defined (up to natural isomorphism) object 
$$
Q=\mathlarger{\circledast}_{\tau\in G} \tau^\ast P\in \mathcal Perv_Y
$$
with natural isomorphisms $Q\to\sigma^\ast Q$ for every $\sigma\in G$
$$
\Psi_\sigma:Q=\mathlarger{\circledast}_{\tau\in G} \tau^\ast P\to\mathlarger{\circledast}_{\tau\in G} (\tau\sigma)^\ast P\cong
$$
$$
\cong\mathlarger{\circledast}_{\tau\in G} \sigma^\ast\tau^\ast P\cong\sigma^\ast\mathlarger{\circledast}_{\tau\in G} \tau^\ast P=\sigma^\ast Q
$$
given by the natural commutation isomorphism and the natural transformation $(\tau\sigma)^\ast\to\sigma^\ast\tau^\ast$, which satisfy the cocycle condition.

\begin{defn}
 Let $P\in\mathcal Perv_Y$ with property $\mathcal P_!$, the {\bf convolution direct image} of $P$ is the element $\pi_{cv\ast}P\in\mathcal Perv_X$ corresponding to the object
 $$
 Q=\mathlarger{\circledast}_{\tau\in G} \tau^\ast P\in \mathcal Perv_Y
 $$
 and the isomorphisms $\Psi_\sigma$ according to the equivalence of categories given in Proposition \ref{equivperv}.
 \end{defn}
 
 \begin{lemma}
  Let $P\in\mathcal Perv_Y$ be a perverse object with property $\mathcal P_!$. Then $\HHH^i_c(\bar Y,P)=0$ for $i\neq 0$.
 \end{lemma}
 
 \begin{proof}
  Let $Q$ be the shifted constant object $A[d]$. Since $Y$ is smooth, $Q$ is a perverse sheaf. Therefore $P\ast_!Q$ is perverse. But $P\ast_!Q$ is the geometrically constant object $P\ast_!A[d]=\R\Gamma_c(Y,P)[d]$, which can only be perverse if it is concentrated in degree $-d$, that is, if $\HHH^i_c(\bar Y,P)=0$ for $i\neq 0$.
 \end{proof}

 If $k\hookrightarrow L$ is a finite Galois extension of fields and $X$ is a commutative group scheme over $\Spec\;k$, then $Y:=X\otimes\Spec\;L$ is a finite Galois étale cover of $X$ via the natural map $\pi:Y\to X$ induced by the projection $\varpi:\Spec\;L\to\Spec\;k$. Given an object $P\in\mathcal Perv_Y$, its cohomology with compact supports $\R\Gamma_c(\bar Y,P)$ with its $\Gal(k^{sep}/L)$-action (where $\bar Y=Y\otimes\Spec\;k^{sep}$) is concentrated in degree $0$ by the previous lemma. Then we have
 
 \begin{proposition}\label{convtensor}
  For every $P\in \mathcal Perv_Y$ we have an isomrphism
  $$
  \HHH^0_c(\bar X,\pi_{cv\ast}P)\cong\varpi_{\otimes\ast}\HHH^0_c(\bar Y,P)
  $$
  in $\mathcal Sh(\Spec\;k,A)$.
 \end{proposition}

 \begin{proof}
  It suffices to show that their pullbacks are isomorphic on $\mathcal{C}_{\Spec\;k,\Spec\;L}$. The pullback of $\varpi_{\otimes\ast}\HHH^0_c(\bar Y,P)$ is, by definition,
  $$
  Q:=\bigotimes_{\tau\in\Gal(L/k)}\tau^\ast\HHH^0_c(\bar Y,P)
  $$
  with the isomorphisms $\Psi_\sigma:Q\to\sigma^\ast Q$ induced by the commutation isomorphisms of the tensor product. On the left hand side, the pullback is
  $$
  \varpi^\ast\HHH^0_c(\bar X,\pi_{cv\ast}P)
  $$
  with the natural isomorphisms
  $$
  \tilde\Psi_\sigma:\varpi^\ast\HHH^0_c(\bar X,\pi_{cv\ast}P)\to\sigma^\ast \varpi^\ast\HHH^0_c(\bar X,\pi_{cv\ast}P)\cong \varpi^\ast\HHH^0_c(\bar X,\pi_{cv\ast}P).
  $$
  By proper base change, there are natural isomorphisms
  $$
  \varpi^\ast\HHH^0_c(\bar X,\pi_{cv\ast}P)\cong \HHH^0_c(\bar Y,\pi^\ast \pi_{cv\ast}P)=\HHH^0_c(\bar Y,\mathlarger{\circledast}_{\tau\in G} \tau^\ast P)
  $$
  and, by the Künneth formula \cite[8.1.8]{katz1990esa},
  $$
  \HHH^0_c(\bar Y,\mathlarger{\circledast}_{\tau\in G} \tau^\ast P)\cong\bigotimes_{\tau\in G}\HHH^0_c(\bar Y,\tau^\ast P)\cong\bigotimes_{\tau\in G}\tau^\ast\HHH^0_c(\bar Y, P)
  $$
  where this isomorphism is compatible with the action of $G$ on both the convolution and the tensor product via permutation of the factors. Similarly, for every $\sigma\in G$ we have
  $$
  \sigma^\ast \varpi^\ast\HHH^0_c(\bar X,\pi_{cv\ast}P)\cong \bigotimes_{\tau\in G}\sigma^\ast\tau^\ast\HHH^0_c(\bar Y, P)
$$  
and, by naturality, the diagrams
$$
\begin{tikzcd}
\varpi^\ast\HHH^0_c(\bar X,\pi_{cv\ast}P) \arrow[r,"\cong"] \arrow[d,"\tilde\Psi_\sigma"] & \bigotimes_{\tau\in G}\tau^\ast\HHH^0_c(\bar Y,P) \arrow[d,"\Psi_\sigma"] \\
\sigma^\ast\varpi^\ast\HHH^0_c(\bar X,\pi_{cv\ast}P) \arrow[r,"\cong"]  & \bigotimes_{\tau\in G}\sigma^\ast\tau^\ast\HHH^0_c(\bar Y,P)
\end{tikzcd}
$$
commute. In other words, the objects $(Q,\{\Psi_\sigma\})$ and $(\varpi^\ast\HHH^0_c(\bar X,\pi_{cv\ast}P),\{\tilde\Psi_\sigma\})$ in $\mathcal{C}_{\Spec\;k,\Spec\;L}$ are isomorphic.
 \end{proof}

 Suppose now that $Y$ and $X$ are commutative group varieties over a finite field $k$, let $\chi:X(k)\to A^\times$ be a character and $\chi \pi:Y(k)\to A^\times$ its pull-back character on $Y$. Let $\LL_\chi$ and $\LL_{\chi \pi}$ the associated rank $1$ locally constant sheaves on $X$ and $Y$ respectively (cf. \cite[1.4-1.8]{deligne569application}).
 
 \begin{proposition}\label{tensorchi}
  For every $P\in\mathcal Perv_Y$ with property $\mathcal P_!$, there is an isomorphism of objects on $\mathcal Perv_X$:
  $$
  \pi_{cv\ast}(P\otimes\LL_{\chi \pi})\cong (\pi_{cv\ast}P)\otimes\LL_\chi
  $$
   \end{proposition}
   
\begin{proof} We will prove the isomorphism by showing that the pullbacks on $\mathcal Perv_{X,Y}$ are isomorphic. The pullback of $\pi_{cv\ast}(P\otimes\LL_{\chi \pi})$ is
$$
\mathlarger{\circledast}_{\tau\in G}\tau^\ast(P\otimes\LL_{\chi \pi})
$$
with the natural isomorphisms $\Psi_\sigma$. On the other hand, the pullback of $(\pi_{cv\ast}P)\otimes\LL_\chi$ is
$$
\pi^\ast((\pi_{cv\ast}P)\otimes\LL_\chi)=(\pi^\ast \pi_{cv\ast}P)\otimes \LL_{\chi \pi}=(\mathlarger{\circledast}_{\tau\in G}\tau^\ast P)\otimes\LL_{\chi \pi}
$$
with the isomorphisms
$$
\tilde\Psi_\sigma:(\mathlarger{\circledast}_{\tau\in G}\tau^\ast P)\otimes\LL_{\chi \pi}\to \sigma^\ast(\mathlarger{\circledast}_{\tau\in G}\tau^\ast P)\otimes\sigma^\ast\LL_{\chi \pi}
$$
induced by the commutation isomorphism $\mathlarger{\circledast}_{\tau\in G}\tau^\ast P\to\sigma^\ast\mathlarger{\circledast}_{\tau\in G}\tau^\ast P$ and the natural isomorphism $\LL_{\chi \pi}\to\sigma^\ast\LL_{\chi \pi}$. Since tensoring with $\LL_{\chi \pi}$ is an auto-equivalence of the tensor category $\Dbc(Y,A)$ (with the convolution operation) by \cite[8.1.10]{katz1990esa} and the commutation isomorphisms are natural, we conclude that the diagrams
$$
\begin{tikzcd}
 \mathlarger{\circledast}_{\tau\in G}\tau^\ast(P\otimes\LL_{\chi \pi}) \arrow[d,"\Psi_\sigma"] \arrow[r,"\cong"] & (\mathlarger{\circledast}_{\tau\in G}\tau^\ast P)\otimes\LL_{\chi \pi} \arrow[d,"\tilde\Psi_\sigma"] \\
 \mathlarger{\circledast}_{\tau\in G}\sigma^\ast\tau^\ast(P\otimes\LL_{\chi \pi})
\arrow[r,"\cong"] & (\mathlarger{\circledast}_{\tau\in G}\sigma^\ast\tau^\ast P)\otimes\sigma^\ast\LL_{\chi \pi}
\end{tikzcd}
$$
are commutative and, in particular, both objects are isomorphic in $\mathcal Perv_{X,Y}$.
\end{proof}

Now let $Y=X\times\Spec\;k_r$ and let $A$ be a finite extension of ${\mathbb Q}_\ell$. For every $y\in X(k_r)=Y(k_r)$, let $\mathrm{N}(y)\in X(k)$ be its norm, defined as $\prod_{\sigma\in G}\sigma(y)$ (which is a point in $X(k_r)$ invariant under $G$, and therefore in $X(k)$)

\begin{proposition}\label{traceformula}
 Under the previos hypothesis, let $P\in\mathcal Perv_Y$ with property $\mathcal P_!$ and let $x\in X(k)$. Then
 $$
 \Tr(F_x|\pi_{cv\ast}P_{\bar x})=\sum_{y\in X(k_r)|\mathrm{N}(y)=x}\Tr(F_y|P_{\bar y})
 $$
\end{proposition}

\begin{proof}
 For every $x\in X(k)$, let $\alpha(x)=\Tr(F_x|\pi_{cv\ast}P)$ and $\beta(x)=\sum_{y\in X(k_r)|\mathrm{N}(y)=x}\Tr(F_y|P)$. View $\alpha$ and $\beta$ as $A$-valued functions on the finite abelian group $X(k)$. Then $\alpha$ and $\beta$ are equal if and only if their discrete Fourier transforms $\hat \alpha$ and $\hat\beta$ are. For every character $\chi:X(k)\to A^\times$ we have
 $$
 \hat\alpha(\chi)=\sum_{x\in X(k)}\chi(x)\Tr(F_x|\pi_{cv\ast}P)=\sum_{x\in X(k)}\Tr(F_x|(\pi_{cv\ast}P)\otimes\LL_\chi)
 $$
 and
 $$
 \hat\beta(\chi)=\sum_{x\in X(k)}\chi(x)\sum_{y\in X(k_r)|\mathrm{N}(y)=x}\Tr(F_y|P)=
 $$
 $$
 =\sum_{y\in X(k_r)}\chi(\mathrm{N}_{k_r/k}(y))\Tr(F_y,P)=\sum_{y\in X(k_r)}\Tr(F_y,P\otimes\LL_{\chi \pi})
 $$
 If $F$ denotes the geometric Frobenius element of $\Gal(\bar k/k)$ then, by the trace formula, we have
 $$
 \hat\alpha(\chi)=\Tr(F|\R\Gamma_c(X,(\pi_{cv\ast}P)\otimes\LL_\chi))=\Tr(F|\HHH^0_c(X,(\pi_{cv\ast}P)\otimes\LL_\chi))
 $$
 and
 $$
 \hat\beta(\chi)=\Tr(F^r|\R\Gamma_c(Y,P\otimes\LL_{\chi \pi}))=\Tr(F^r|\HHH^0_c(Y,P\otimes\LL_{\chi \pi}))
 $$
 and, by propositions \ref{tracetensorind}, \ref{convtensor} and \ref{tensorchi}: 
 $$
 \hat\alpha(\chi)=\Tr(F|\HHH^0_c(X,(\pi_{cnv\ast}P)\otimes\LL_\chi)=\Tr(F|\HHH^0_c(X,\pi_{cnv\ast}(P\otimes\LL_{\chi \pi}))=
 $$
 $$
 =\Tr(F|\pi_{\otimes\ast}\HHH^0_c(Y,P\otimes\LL_{\chi \pi}))=\Tr(F^r|\HHH^0_c(Y,P\otimes\LL_{\chi \pi}))=\hat\beta(\chi).
 $$
\end{proof}

\section{Applications and examples}

In this section we will describe how the convolution direct image construction can be applied to obtain improved estimates for the number of rational points on some varieties over finite fields or for certain exponential sums. We will give two important examples to illustrate this:  Artin-Schreier curves and superelliptic curves. These examples were worked out directly in \cite{rl2010number}, here we will show how they can be obtained as a simple application of our convolution direct image construction.

In order to use this construction for obtaining estimates for the number of rational points on varieties or for partial exponential sums, we need a way to control the weight of the convolution direct image of an object. Suppose again that $X$ is a commutative group variety over a finite field $k$ and $Y=X\times\Spec\;k_r$. We will fix an embedding $\iota:\Ql\to\CC$ so we can freely spreak about weights of $\ell$-adic objects (by which we will mean $\iota$-weights for the given embedding $\iota$). Let $q=p^n$ denote the cardinality of $q$.

\begin{proposition}\label{weights}
 Let $P\in\mathcal Perv(Y,\Ql)$ be an object mixed of weights $\leq w$ with property $\mathcal P_!$. Then $\pi_{cv\ast}P\in\Dbc(X,\Ql)$ is mixed of weights $\leq rw$.
\end{proposition}

\begin{proof}
 Since being mixed is invariant under finite extension of the base field, it suffices to prove it for $\pi^\ast \pi_{cv\ast}P$. But 
 $$
 \pi^\ast \pi_{cv\ast}P\cong P\ast_! (F^\ast P) \ast_! \cdots (F^{(r-1)\ast}P)
 $$
 so this is a consequence of \cite[5.1.14]{beilinson1982faisceaux}.
\end{proof}

\begin{cor}\label{estimate}
 Suppose that $P$ is mixed of weights $\leq w$. Let $x\in X(k)$, and suppose that ${\mathcal H}^i(P\ast_!(F^\ast P)\ast_!\cdots\ast_!(F^{(r-1)\ast}P))_{\bar x}=0$ for $i>n$. Then
 $$
 \left|\sum_{y\in X(k_r);\mathrm{N}(y)=x}\Tr(F_y|P_{\bar y})\right|\leq C\cdot q^{(rw+n)/2}
 $$
 where $C=\sum_i\dim\;{\mathcal H}^i(P\ast_!(F^\ast P)\ast_!\cdots\ast_!(F^{(r-1)\ast}P))_{\bar x}$.
\end{cor}

\begin{proof}
 By Proposition \ref{traceformula}, we have
 $$
 \left|\sum_{y\in X(k_r)|\mathrm{N}(y)=x}\Tr(F_y|P_{\bar y})\right|=|\Tr(F_x|(\pi_{cv\ast}P)_{\bar x})|=\left|\sum_i(-1)^i\Tr(F_x|({\mathcal H}^i(\pi_{cv\ast}P))_{\bar x})\right|\leq
 $$
 $$
 \leq\sum_i|\Tr(F_x|({\mathcal H}^i(\pi_{cv\ast}P))_{\bar x})|\leq\sum_i(\dim\;{\mathcal H}^i(\pi_{cv\ast}P)_{\bar x})q^{(rw+i)/2}\leq
 $$
 $$
 \leq \left(\sum_i\dim\;{\mathcal H}^i(\pi_{cv\ast}P)_{\bar x}\right)q^{(rw+n)/2}=C\cdot q^{(rw+n)/2}
 $$
 since $\pi_{cv\ast}P$ is mixed of weights $\leq rw$, so ${\mathcal H}^i(\pi_{cv\ast}P)$ is mixed of weights $\leq rw+i$ for every $i$.
 \end{proof}

Let us now specialize to the case $X=\AAA^1_k$.

\begin{proposition}\label{artin-schreier}
 Let $f\in k_r[x]$ be a polynomial of degree $d$ prime to $p$, and let $\FF$ be the $\Ql$-sheaf on $\AAA^1_{k_r}$ given by the kernel of the trace morphism $f_\ast\Ql\to\Ql$. Let $S\subseteq\bar k$ be the set of critical values of $f$, and suppose that $a\notin S+F(S)+\cdots+F^{r-1}(S)$ (where $F$ is the Frobenius automorphism $x\mapsto x^q$ of $\bar k$). Then
 $$
 \left|\sum_{y\in k_r;\Tr(y)=a}\Tr(F_y|{\mathcal F}_{\bar y})\right|\leq (d-1)^r \cdot q^{(r-1)/2}.
 $$
\end{proposition}

\begin{proof}
 Let $P=\FF[1]\in\Dbc(Y,\Ql)$. Then $P$ is a perverse sheaf, which is lisse on $\AAA^1_{k_r}\backslash S$ and tamely ramified at infinity (and, in fact, its monodromy at infinity is the direct sum of all non-trivial characters with trivial $d$-th power). In particular, it has property $\mathcal P_!$, since it does not have any quotient isomorphic to an Artin-Schreier object \cite[Lemma 2.6.13, 2.6.14]{katz1996rls}. By Laumon's local Fourier transform theory \cite[2.4]{laumon1987transformation}, its Fourier transform (with respect to a fixed additive character of $k$) is of the form $\GGG[1]$, where $\GGG$ is lisse on $\GG_m$ of rank $d-1$, tamely ramified at zero and whose monodromy at infinity is a direct sum of representations of the form $\LL_{\psi(tx)}\otimes{\mathcal K}_t$ for $t\in S$, where $\LL_{\psi(tx)}$ is an Artin-Schreier character and ${\mathcal K}_t$ has slopes $<1$.

 Similarly, for every integer $i$ the Fourier transform $\GGG_i$ of $F^{i\ast}P$ is lisse on $\GG_m$ of rank $d-1$, tamely ramified at $0$, and its monodromy at infinity is a direct sum of representations of the form $\LL_{\psi(t^{q^i}x)}\otimes{\mathcal K}_{t^{q^i}}$. Since Fourier transform interchanges convolution and tensor product (up to a shift), the Fourier transform of $P\ast_! (F^\ast P)\ast_!\cdots\ast_!(F^{(r-1)\ast}P)$ is $\GGG_0\otimes\GGG_1\otimes\cdots\otimes\GGG_{r-1}[1]$.
 
 This tensor product is lisse on $\GG_m$ of rank $(d-1)^r$, tamely ramified at zero, and its monodromy at infinity is a direct sum of representations of the form ${\mathcal K}\otimes\LL_{\psi((t_0+t_1^q+\cdots+t_{r-1}^{q^{r-1}})x)}$, where $\mathcal K$ has slopes $<1$. In particular, its tensor product with $\LL_{\psi(-ax)}$ is totally wild at infinity, since $a\notin S+F(S)+\cdots+F^{r-1}(S)$. Again by local Fourier transform theory, this implies that $P\ast_! (F^\ast P)\ast_!\cdots\ast_!(F^{(r-1)\ast}P)$ is concentrated in degree $-1$ and unramified at $a$, and its rank there is the dimension of $\HHH^1_c(\AAA^1_{\bar k},\GGG_0\otimes\GGG_1\otimes\cdots\otimes\GGG_{r-1}\otimes\LL_{\psi(ax)})$, which by the Ogg-Shafarevic formula is $(d-1)^r$ (its Swan conductor at infinity).
 
 We conclude that $\HH^i(P\ast_!(F^\ast P)\ast_!\cdots\ast_!(F^{(r-1)\ast}P))_a=0$ for $i>-1$. Since $\FF$ is pure of weight $0$, $P$ is mixed of weights $\leq 1$. Then $\pi_{cv\ast}P$ is mixed of weights $\leq r$ by Proposition \ref{weights}. Applying corollary \ref{estimate}, we get
 $$
 \left|\sum_{y\in k_r;\Tr(y)=a}\Tr(F_y|{\mathcal P}_{\bar y})\right|
 =
 \left|\sum_{y\in k_r;\Tr(y)=a}\Tr(F_y|{\mathcal F}_{\bar y})\right|\leq (d-1)^r \cdot q^{(r-1)/2}
 $$
 \end{proof}
 
 Applying this to $a=0$ and using the fact that the number of solutions in $k_r$ to the equation $y^q-y=t$ is $q$ if $\Tr_{k_r/k}(t)=0$ and $0$ otherwise, we get  

 \begin{cor}
  Under the previous hypotheses, the number $N$ of $k_r$-valued points on the Artin-Schreier curve $y^q-y=f(x)$ is bounded by
  $$
  |N-q^r|\leq (d-1)^r\cdot q^{(r+1)/2}.
  $$
   \end{cor}

 \begin{proof}
  The number $N$ can be expressed as
  $$
  N=\sum_{x\in k_r}\#\{y\in k_r;y^q-y=f(x)\}=q\cdot\#\{x\in k_r;\Tr_{k_r/k}(f(x))=0\}=
  $$
  $$
  =q\cdot \sum_{y\in k_r;\Tr(y)=0}\#\{x\in k_r;f(x)=y\}=q\cdot\sum_{y\in k_r;\Tr(y)=0}\Tr(F_y|(f_\ast\Ql)_{\bar y})=
  $$
  $$
  =q\cdot\left(q^{r-1}+\sum_{y\in k_r;\Tr(y)=0}\Tr(F_y|\FF_{\bar y})\right)=q^r+q\sum_{y\in k_r;\Tr(y)=0}\Tr(F_y|\FF_{\bar y})
  $$
  and we just need to apply the previous proposition.
 \end{proof}

 We now study the same example on the multiplicative group ${\mathbb G}_{m,k}$.
 
 \begin{proposition}
 Let $f\in k_r[x]$ be a polynomial of degree $d$ prime to $p$ without multiple roots and such that $f'$ has no roots with multiplicity $\geq p$, and let $\FF$ be the $\Ql$-sheaf on ${\mathbb G}_{m,k}$ given by the kernel of the trace morphism $f_\ast\Ql\to\Ql$. Let $S\subseteq\bar k$ be the set of critical values of $f$, and suppose that $a\notin S\cdot F(S)\cdot\cdots\cdot F^{r-1}(S)$ (where $F$ is the Frobenius automorphism $x\mapsto x^q$ of $\bar k$). Then
 $$
 \left|\sum_{y\in k_r^\times;\mathrm{N}(y)=a}\Tr(F_y|{\mathcal F}_{\bar y})\right|\leq r(d-1)^r \cdot q^{(r-1)/2}.
 $$
\end{proposition}

\begin{proof}
 Since $f$ does not have multiple roots, $0$ is not a critical value of $f$. This implies that $\FF$ is unramified at $0$. Then $\FF$ has trivial monodromy at $0$ and its monodromy at infinity is the direct sum of all non-trivial characters of the inertia group of ${\mathbb G}_{m,k}$ at $\infty$ with trivial $d$-th power. In particular, $\FF$ does not have any subquotient isomorphic to a Kummer sheaf $\LL_\chi$ (since such a sheaf would have non-trivial monodromy at zero if $\chi$ is not trivial, and trivial monodromy at infinity otherwise). Then the perverse sheaf $P:=\FF[1]$ has property $\mathcal P_!$, and the same happens to $F^{i\ast}P$ for every integer $i$. Therefore
 $$
 P\ast_! (F^\ast P)\ast_!\cdots\ast_!(F^{(r-1)\ast}P)
 $$
 is a perverse sheaf on ${\mathbb G}_{m,k}$ by \cite[2.6]{katz1996rls}. By \cite[Lemma 19.5]{katz2012convolution}, this sheaf is smooth on ${\mathbb G}_{m,k}\backslash T$, where $T=S\cdot F(S)\cdots F^{r-1}(S)$. So, if $a\notin T$, we have ${\mathcal H}^i(P\ast_! (F^\ast P)\ast_!\cdots\ast_!(F^{(r-1)\ast}P))_a={\mathcal H}^{r+i}(\FF\ast_! (F^\ast \FF)\ast_!\cdots\ast_!(F^{(r-1)\ast}\FF))_a=0$ for $i>-1$. Since $\FF$ is pure of weight $0$, $P$ is mixed of weights $\leq 1$, and by corollary \ref{estimate} we get the estimate
 $$
 \left|\sum_{y\in k_r^\times;\mathrm{N}(y)=a}\Tr(F_y|{\mathcal F}_{\bar y})\right|\leq C \cdot q^{(r-1)/2}.
 $$
 where $C$ is the rank of ${\mathcal H}^{-1}(P\ast_! (F^\ast P)\ast_!\cdots\ast_!(F^{(r-1)\ast}P))$ at $a$. 
 
 We claim that this rank is $r(d-1)^r$. We will prove this by induction on $r$: for $r=1$ it is obvious. Suppose that we have shown that $Q:=P\ast_! (F^\ast P)\ast_!\cdots\ast_!(F^{(r-2)\ast}P)$ is a perverse sheaf whose ${\mathcal H}^{-1}$ has generic rank $(r-1)(d-1)^{r-1}$. The hypotheses on $f$ imply that $\FF$ (and so $F^{i\ast}\FF$ for every $i$) is everywhere tamely ramified. Therefore so is $\HH^{-1}(Q)$ by \cite[Corollary 24]{rojas2013local}. The rank of $\HH^{-1}(Q\ast_!(F^{(r-1)\ast}P))$ at $a$ is then
 $$
 \dim\HHH^1_c({\mathbb G}_{m,\bar k},\HH^{-1}(Q)\otimes \tau_a^\ast F^{(r-1)\ast}\FF)
 $$
 where $\tau_a:{\mathbb G}_{m,\bar k}\to{\mathbb G}_{m,\bar k}$ is the map $t\mapsto a/t$. By the Ogg-Shafarevic formula, since $\HH^{-1}(Q)\otimes \tau_a^\ast F^{(r-1)\ast}\FF$ is everywhere tamely ramified and $\HHH^1_c$ is its only non-zero cohomology group, this dimension is given by the sum of the rank drops at the ramification points. 
 
 Now, since $a\notin S\cdot F(S)\cdot\cdots\cdot F^{r-1}(S)$, the sets of ramification points of $\HH^{-1}(Q)$ (which is contained in $S\cdot F(S)\cdots F^{r-2}(S)$) and $\tau_a^\ast F^{(r-1)\ast}\FF$ are disjoint, so the set of ramification points of $\HH^{-1}(Q)\otimes \tau_a^\ast F^{(r-1)\ast}\FF$ is the union of these two, and the drop at each ramification point coming from one of the factors is the drop of the point in said factor multiplied by the generic rank of the other factor. That is,
 $$
 \dim\HHH^1_c({\mathbb G}_{m,\bar k},\HH^{-1}(Q)\otimes \tau_a^\ast F^{(r-1)\ast}\FF)=
 $$
 $$
 =\mathrm{rank}(\HH^{-1}(Q))\dim\HHH^1_c({\mathbb G}_{m,\bar k},\tau_a^\ast F^{(r-1)\ast}\FF)+
 $$
 $$
 +\mathrm{rank}(\tau_a^\ast F^{(r-1)\ast}\FF)\dim\HHH^1_c({\mathbb G}_{m,\bar k},\HH^{-1}(Q))
 $$
 $$
 (r-1)(d-1)^{r-1}(d-1)+(d-1)(d-1)^{r-1}=r(d-1)^r.
 $$
 since
 $$
 \dim\HHH^1_c({\mathbb G}_{m,\bar k},\HH^{-1}(Q))=\chi_c({\mathbb G}_{m,\bar k},Q)=\chi_c({\mathbb G}_{m,\bar k},P\ast_! (F^\ast P)\ast_!\cdots\ast_!(F^{(r-2)\ast}P))=
 $$
 $$
 =\chi_c({\mathbb G}_{m,\bar k},P)\chi_c({\mathbb G}_{m,\bar k},F^\ast P)\cdots\chi_c({\mathbb G}_{m,\bar k},F^{(r-2)\ast}P)=\chi_c({\mathbb G}_{m,\bar k},P)^{r-1}
 $$
and
$$
\chi_c({\mathbb G}_{m,\bar k},P)=-\chi_c({\mathbb G}_{m,\bar k},\FF)=(d-1)-\chi_c(\AAA^1_{\bar k},\FF)
$$
since $f$ has no multiple roots (so $\FF$ has rank $d-1$ at $0$), and
$$
\chi_c(\AAA^1_{\bar k},\FF)=\chi_c(\AAA^1_{\bar k},f_\ast\Ql)-\chi_c(\AAA^1_{\bar k},\Ql)=1-1=0.
$$
\end{proof}

Applying this to $a=1$ and using the fact that the number of solutions in $k_r$ to the equation $y^{q-1}=t$ is $q-1$ if $\mathrm{N}_{k_r/k}(t)=1$ and $0$ otherwise, we get  

 \begin{cor}
  Under the previous hypotheses, the number $N$ of $k_r$-valued points on the superelliptic curve $y^{q-1}=f(x)$ is bounded by
  $$
  |N-q^r-\delta+1|\leq r(d-1)^r(q-1)\cdot q^{(r-1)/2}
  $$
  where $\delta$ is the number of roots of $f$ in $k_r$.
   \end{cor}

 \begin{proof}
  The number $N$ can be expressed as
  $$
  N=\sum_{x\in k_r}\#\{y\in k_r;y^{q-1}=f(x)\}=\delta+\sum_{x\in k_r}\#\{y\in k_r^\times;y^{q-1}=f(x)\}=
  $$
  $$
  =\delta+(q-1)\cdot\#\{x\in k_r;\mathrm{N}_{k_r/k}(f(x))=1\}
  =
  $$
  $$
  =\delta+(q-1)\cdot \sum_{y\in k_r^\times;\mathrm{N}(y)=1}\#\{x\in k_r;f(x)=y\}=
  $$
  $$
  =\delta+(q-1)\cdot\sum_{y\in k_r^\times;\mathrm{N}(y)=1}\Tr(F_y|(f_\ast\Ql)_{\bar y})=
  $$
  $$
  =\delta+(q-1)\cdot\left(\frac{q^r-1}{q-1}+\sum_{y\in k_r^\times;\mathrm{N}(y)=1}\Tr(F_y|\FF_{\bar y})\right)=
  $$
  $$
  =\delta+q^r-1+(q-1)\sum_{y\in k_r^\times;\mathrm{N}(y)=1}\Tr(F_y|\FF_{\bar y})
  $$
  and we just need to apply the previous proposition.
 \end{proof}
 
 To conclude, let us compare the estimates obtained using the convolution direct image with those obtained in \cite{rl-rationality} using the convolution Adams operation. For objects which are defined on $X$, the Adams operation generally gives better estimates (this is to be expected, since the Adams operation is a linear combination of subobjects of the convolution power, so it has lower rank). For instance, if $g\in k[x]$ and $a\in k$ satisfy the hypotheses of Proposition \ref{artin-schreier}, by \cite[Example 6.5]{rl-rationality} we have
 $$
 \left|\#\{x\in k_r;\Tr(f(x))=a\}-q^{r-1}\right|\leq C_{d,r} \cdot q^{(r-1)/2}.
 $$
 where $C_{d,r}=\sum_{i=0}^{r-1}{\binom{d+r-i-2}{r}}\binom{r-1}{i}$, which is better than the estimate in Proposition \ref{artin-schreier}, as this table shows (for $d=5$):
\begin{center}
\begin{tabular}{r|r|r}
$r$ & $C_{d,r}$ & $(d-1)^r$ \\ \hline
2 & 16 & 16\\
3 & 44 & 64\\
4 & 96 & 256 \\
5 & 180 & 1024 \\
10 & 1360 & 1048576 \\
20 & 10720 & 1099511627776
\end{tabular}
\end{center}

On the other hand, for objects which are not defined on $X$ (only on $Y$), the method in \cite[Section 8]{rl-rationality} implies taking the direct sum of all its Galois conjugates (which descends to $X$) and then taking convolution Adams power, which greatly increases the rank. In this case, the estimates here are much better: applying \cite[Corollary 8.2]{rl-rationality}, for $f\in k_r[x]\backslash k[x]$ we would get
$$
 \left|\#\{x\in k_r;\Tr(f(x))=a\}-q^{r-1}\right|\leq \frac{1}{r}C_{rd-r+1,r} \cdot q^{(r-1)/2}.
$$
which is worse than the estimate in \ref{artin-schreier}, as this table shows (again for $d=5$):
\begin{center}
\begin{tabular}{r|r|r}
$r$ & $\lceil\frac{1}{r}C_{rd-r+1,r}\rceil$ & $(d-1)^r$ \\ \hline
2 & 32 & 16\\
3 & 386 & 64\\
4 & 5504 & 256 \\
5 & 86401 & 1024 \\
10 & 153547568007 & 1048576 \\
20 & 1356608411506872363943501 & 1099511627776
\end{tabular}
\end{center}

So the estimates provided by the convolution direct image greatly improve those of \cite{rl-rationality} for objects not defined on $X$. On the other hand, note that the convolution direct image operation does \emph{not} commute with finite extensions of the base field, so this operation can not be used to construct a ``Trace $L$-function'' as if was done in \cite{rl-rationality} for objects defined on $X$ (and, in fact, such a function is not rational in general, as it was shown in \cite[Section 8]{rl-rationality}).

\bibliographystyle{amsalpha}
\bibliography{bibliography}

\end{document}